\documentclass[11pt, english]{article}

\usepackage[margin= 2 cm,bottom=21mm, top= 20mm]{geometry}

\usepackage[dvipsnames]{xcolor}
\usepackage[square,sort,comma,numbers]{natbib}
\usepackage{amsthm}
\usepackage{amsmath}
\usepackage{amssymb}
\usepackage{setspace}
\usepackage{mathtools}
\usepackage{graphicx}
\graphicspath{ {./images/} }
\usepackage[hidelinks]{hyperref}
\usepackage{cleveref}
\usepackage{hyperref}
\usepackage{enumitem}
\usepackage{framed}
\usepackage{subcaption}
\usepackage{xspace}
\usepackage{thmtools} 
\usepackage{thm-restate}
\usepackage{tikz-feynman}
\usepackage{thmtools}
\usepackage{thm-restate}
\usepackage{ifthen} 
\usepackage{floatrow}

\usepackage{tikz}
\usepackage{mathdots}
\usepackage{xcolor}
\usepackage{diagbox}
\usepackage{colortbl}
\usepackage[absolute,overlay]{textpos}

\usetikzlibrary{shapes.misc}
\usetikzlibrary{decorations.pathmorphing}

\tikzset{snake it/.style={decorate, decoration=snake}}
\usetikzlibrary{math}

\usetikzlibrary{calc}
\usetikzlibrary{decorations.pathreplacing}
\usetikzlibrary{positioning,patterns}
\usetikzlibrary{arrows,shapes,positioning}
\usetikzlibrary{decorations.markings}

\tikzstyle{edge}=[very thick]
\definecolor{bostonuniversityred}{rgb}{0.8, 0.0, 0.0}
\definecolor{arsenic}{rgb}{0.23, 0.27, 0.29}
\tikzstyle{diredge}=[postaction={decorate,decoration={markings,
		mark=at position .95 with {\arrow[scale = 1]{stealth};}}}]
\tikzset{
    arrow/.style={decoration={markings, mark=at position 0.7 with
    {\fill(-0.09*#1,-0.03*#1) -- (0,0) -- (-0.09*#1,0.03*#1) -- cycle;}}, postaction={decorate}},
    arrow/.default=1
}
\tikzset{
    arow/.style={decoration={markings, mark=at position 1 with
    {\fill(-0.09*#1,-0.03*#1) -- (0,0) -- (-0.09*#1,0.03*#1) -- cycle;}}, postaction={decorate}},
    arow/.default=1
}
\tikzset{
    arrrow/.style={decoration={markings, mark=at position 0.9 with
    {\fill(-0.09*#1,-0.03*#1) -- (0,0) -- (-0.09*#1,0.03*#1) -- cycle;}}, postaction={decorate}},
    arow/.default=1
}

\newcommand{\fitellipsis}[2] 
{\draw [fill=white]let \p1=(#1), \p2=(#2), \n1={atan2(\y2-\y1,\x2-\x1)}, \n2={veclen(\y2-\y1,\x2-\x1)}
    in ($ (\p1)!0.5!(\p2) $) ellipse [ x radius=\n2/2+0cm, y radius=1.1cm, rotate=\n1];
}
\newcommand{\Fitellipsis}[2] 
{\draw [fill=white]let \p1=(#1), \p2=(#2), \n1={atan2(\y2-\y1,\x2-\x1)}, \n2={veclen(\y2-\y1,\x2-\x1)}
    in ($ (\p1)!0.5!(\p2) $) ellipse [ x radius=\n2/2+0cm, y radius=1.4cm, rotate=\n1];
}

\theoremstyle{plain}

\newtheorem*{thm*}{Theorem}
\newtheorem{thm}{Theorem}[section]
\Crefname{thm}{Theorem}{Theorems}

\newtheorem*{lem*}{Lemma}
\newtheorem{lem}[thm]{Lemma}
\Crefname{lem}{Lemma}{Lemmas}

\newtheorem*{claim*}{Claim}
\newtheorem{claim}[thm]{Claim}
\crefname{claim}{Claim}{Claims}
\Crefname{claim}{Claim}{Claims}

\Crefname{prop}{Proposition}{Propositions}

\Crefname{remar}{Remark}{Remarks}

\crefname{cor}{Corollary}{Corollaries}

\newtheorem*{conj*}{Conjecture}

\crefname{conj}{Conjecture}{Conjectures}

\Crefname{qn}{Question}{Questions}

\Crefname{obs}{Observation}{Observations}

\Crefname{ex}{Example}{Examples}

\theoremstyle{definition}

\Crefname{prob}{Problem}{Problems}

\newtheorem{defn}[thm]{Definition}
\Crefname{defn}{Definition}{Definitions}

\theoremstyle{remark}

\renewenvironment{proof}[1][]{\begin{trivlist}
\item[\hspace{\labelsep}{\bf\noindent Proof#1.\/}] }{\qed\end{trivlist}}

\newcommand{\remove}[1]{}

\title{\vspace{-0.85 cm}
A generalization of Bondy's pancyclicity theorem}
\date{}

\author{
Nemanja Dragani\'c\thanks{
Department of Mathematics, ETH, Z\"urich, Switzerland. Research supported in part by SNSF grant 200021\_196965.
\newline
\emph{Emails}: \textbf{\{nemanja.draganic,david.munhacanascorreia, benjamin.sudakov\}@math.ethz.ch}.
}
\and
David Munh\'a Correia\footnotemark[1]
\and
Benny Sudakov\footnotemark[1]}
\begin{document} 
\maketitle
\begin{abstract}
The \emph{bipartite independence number} of a graph $G$, denoted as $\tilde\alpha(G)$, is the minimal number $k$ such that there exist positive integers $a$ and $b$ with $a+b=k+1$ with the property that for any two sets $A,B\subseteq V(G)$ with $|A|=a$ and $|B|=b$, there is an edge between $A$ and $B$.
McDiarmid and Yolov showed that if $\delta(G)\geq\tilde \alpha(G)$ then $G$ is Hamiltonian, extending the famous theorem of Dirac which states that if $\delta(G)\geq  |G|/2$ then $G$ is Hamiltonian.
In 1973, Bondy showed that, unless $G$ is a complete bipartite graph, Dirac's Hamiltonicity condition also implies pancyclicity, i.e., existence of cycles of all the lengths from $3$ up to $n$.
In this paper we show that $\delta(G)\geq\tilde \alpha(G)$ implies that $G$ is pancyclic or that $G=K_{\frac{n}{2},\frac{n}{2}}$, thus extending the result of McDiarmid and Yolov, and generalizing the classic theorem of Bondy.
\end{abstract}

\section{Introduction}

The notion of Hamiltonicity is one of most central and extensively studied topics in Combinatorics. Since the problem of determining whether a graph is Hamiltonian is NP-complete, a central theme in Combinatorics is to derive sufficient conditions for this property. A classic example is Dirac’s theorem \cite{dirac1952some} which dates back to 1952 and states that every $n$-vertex graph with minimum degree at least $n/2$ is Hamiltonian. Since then, a plethora of interesting and important results about various aspects of Hamiltonicity have been obtained, see e.g. \cite{ajtai1985first,chvatal1972note,kuhn2013hamilton,krivelevich2011critical,krivelevich2014robust,MR3545109,ferber2018counting, cuckler2009hamiltonian, posa1976hamiltonian}, and the surveys \cite{gould2014recent, MR3727617}.

Besides finding sufficient conditions for containing a Hamilton cycle, significant attention has been given to conditions which force a graph to have cycles of other lengths. Indeed, \emph{the cycle spectrum of a graph}, which is the set of lengths of cycles contained in that graph, has been the focus of study of numerous papers and in particular gained a lot of attention in recent years \cite{liu2020solution, alon2022cycle, friedman2021cycle, liu2018cycle, verstraete2016extremal, alon2021divisible, milans2012cycle, keevash2010pancyclicity, ourpaper,bucic2021cycles}. Among other graph parameters, the relation of the cycle spectrum to the minimum degree, number of edges, independence number, chromatic number and expansion of the graph have been studied.

We say that an $n$-vertex graph is \emph{pancyclic} if the cycle spectrum contains all integers from $3$ up to $n$. Bondy suggested that in the cycle spectrum of a graph, it is usually hardest to guarantee the existence of the longest cycle, i.e. a Hamilton cycle.  This intuition was captured by his famous meta-conjecture \cite{bondy10pancyclic} from 1973, which asserts that any non-trivial condition which implies Hamiltonicity, also implies pancyclicity (up to a small class of exceptional graphs). As a first example, he proved in \cite{bondy1971pancyclic} an extension of Dirac's theorem, showing that minimum degree at least $n/2$ implies that the graph is either pancyclic or that it is the complete bipartite graph $K_{\frac{n}{2},\frac{n}{2}}$. Further, Bauer and Schmeichel \cite{bauer1990hamiltonian}, relying on previous results of Schmeichel and Hakimi \cite{schmeichel1988cycle}, showed that the sufficient conditions for Hamiltonicity given by Bondy \cite{bondy1980longest}, Chvátal \cite{chvatal1972hamilton} and Fan \cite{fan1984new} all imply pancyclicity, up to a certain small family of exceptional graphs.

Another classic Hamiltonicity result is the Chv\'atal-Erd\H{o}s theorem, which states that $\kappa(G)\geq \alpha(G)$ implies that $G$ is Hamiltonian, where $\kappa(G)$ is the connectivity of $G$, and $\alpha(G)$ its independence number. Motivated by Bondy's meta-conjecture, Jackson and Ordaz \cite{jackson1990chvatal} thirty years ago suggested that $\kappa(G)> \alpha(G)$ already implies pancyclicity. The first progress towards this problem was obtained by Keevash and Sudakov, who showed pancyclicity when $\kappa(G)\geq 600\alpha(G)$. Recently, in \cite{ourJOpaper} we were able to resolve the
Jackson-Ordaz conjecture asymptotically, proving that 
$\kappa(G)\geq (1+o(1))\alpha(G)$ is already enough for pancyclicity. It is worth mentioning that, in all the listed work, the proof that the Hamiltonicity condition also implies pancyclicity is usually significantly harder than just proving Hamiltonicity, and requires new ideas and techniques.

An interesting sufficient condition for Hamiltonicity was given by McDiarmid and Yolov \cite{mcdiarmid2017hamilton}. To state their result, we need the following natural graph parameter. For a graph $G$, its \emph{bipartite independence number} $\tilde\alpha(G)$ is the minimal number $k$, such that there exist positive integers $a$ and $b$ with $a+b=k+1$, such that between any two sets $A,B\subseteq V(G)$ with $|A|=a$ and $|B|=b$, there is an edge between $A$ and $B$. Notice that we always have that $\alpha(G)\leq \tilde \alpha (G)$. Indeed, if $\tilde\alpha(G)=k$, then $G$ does not contain independent sets $I$ of size at least $k+1$, since evidently for every $a+b=k+1$, there would exist disjoint sets $A,B\subset I$, so that $|A|=a$ and $|B|=b$ and with no edge between $A$ and $B$. Let us now state the result of McDiarmid and Yolov.

\begin{thm}[\cite{mcdiarmid2017hamilton}]\label{thm:mcdiarmid}
If $\delta(G)\geq \tilde\alpha(G)$, then $G$ is Hamiltonian.
\end{thm}
\noindent This result implies Dirac's theorem, because if $\delta(G)\geq n/2$, then $\lceil n/2\rceil \geq \tilde \alpha(G)$, as for every $|A|=1$ and $|B|=\lceil n/2\rceil$ there is an edge between $A$ and $B$. Hence also $\delta(G)\geq\lceil n/2\rceil\geq \tilde\alpha(G)$, so $G$ is Hamiltonian.
 
Naturally, the immediate question which arises is whether the McDiarmid-Yolov condition implies that the graph satisfies the stronger property of pancyclicity. As a very preliminary step in this direction, Chen~ \cite{chen2022hamilton} was able to show that for any given positive constant $c$, for sufficiently large $n$ it holds that if $G$ is an $n$-vertex graph with $\tilde\alpha(G)=cn$ and $\delta(G)\geq \frac{10}{3}cn$, then $G$ is pancyclic.
In this paper we completely resolve this problem, showing that $\delta(G)\geq\tilde \alpha(G)$ implies that $G$ pancyclic  or $G=K_{\frac{n}{2},\frac{n}{2}}$. This generalizes the classical theorem of Bondy \cite{bondy1971pancyclic}, and gives additional evidence for his meta-conjecture, mentioned above.

\begin{thm}\label{thm:main}
If $\delta(G)\geq \tilde\alpha(G)$, then $G$ is pancyclic, unless $G$ is complete bipartite $G=K_{\frac{n}{2},\frac{n}{2}}$.
\end{thm}

\noindent Our proof is completely self-contained and relies on a novel variant of P\'osa's celebrated rotation-extension technique, which is used to extend paths and cycles in expanding graphs (see, e.g., \cite{posa1976hamiltonian}). Define the graph $\tilde C_\ell$, to be the cycle of length $\ell$ together with an additional vertex which is adjacent to two consecutive vertices on the cycle (thus forming a triangle with them).
For each $\ell\in[3,n-1]$, our goal is to either find a $\tilde C_\ell$ or a $\tilde C_{\ell+1}$, which is clearly enough to show pancyclicity. 
The proof is recursive in nature, as we will derive the existence of a $\tilde C_\ell$ or a $\tilde C_{\ell+1}$ from the existence of a $\tilde C_{\ell-1}$. In our setting, we would like to apply the rotation-extension technique to the $\tilde C_{\ell-1}$ with the additional requirement that the extended cycle preserves the attached triangle. However, this is not possible in general and from the existence of a $\tilde C_{\ell-1}$ we will in turn derive the existence of a gadget denoted as a \emph{switch}, which is a path with triangles attached to it, to which we can apply our rotation-extension technique. 
One of the key ideas is to consider the switch which is optimal with respect to how close the triangles are to the beginning of the path (see \Cref{def:path with triangle}).
The application of the rotation-extension technique to such an optimal switch will then  result in either a $\tilde C_\ell$, a $\tilde C_{\ell+1}$, or a better switch, contradicting the optimality of the original switch. The details are given in the next section.

\section{The proof}
We first recall the definition of $\tilde\alpha(G)$.
\begin{defn}
For a graph $G$, let $\tilde\alpha(G)$ denote the minimal number $k$, such that there exist positive integers $a$ and $b$ with $a+b=k+1$, such that between any two disjoint sets $A,B\subseteq V(G)$ with $|A|=a$ and $|B|=b$, there is an edge between $A$ and $B$.
\end{defn} 
\noindent We will also need the following definition of a cycle which has one triangle attached to one of its edges.
\begin{defn}

Define the graph $\tilde C_\ell$, to be the cycle of length $\ell$ together with an additional vertex which is adjacent to two consecutive vertices on the cycle.
\end{defn}

\begin{proof}[ of Theorem~\ref{thm:main}]
Let $n :=|V(G)|$ and denote $k :=\tilde\alpha(G)$ and suppose that for $a\leq b$ and $a+b=k+1$, between every two disjoint vertex sets of sizes $a$ and $b$, there is an edge between them. By assumption, we have $\delta(G)\geq k$. Note also that as observed before $G$ has independence number $\alpha(G) \leq k$. Note further that since $\delta(G) \geq \alpha(G)$, we have that if $G$ is bipartite then it must be isomorphic to $K_{n/2,n/2}$. Finally, note also that $G$ is connected. Indeed, consider two non-adjacent vertices $u,v$; if their neighbourhoods intersect, there clearly exists a $uv$-path; otherwise, since both neighbourhoods have size at least $\delta(G) \geq k \geq a,b$, there exists an edge between them and thus also a $uv$-path.

\begin{claim}\label{claim:triangle}
Either $G$ contains a triangle or $G$ is bipartite.
\end{claim}
\begin{proof}
For sake of contradiction, suppose it does not contain a triangle nor it is bipartite and consider any vertex $v\in G$. If its neighbourhood is of size at least $k+1$, then as observed above it contains an edge, which together with $v$ creates a triangle. Therefore, every vertex has degree $k$ and its neighbourhood does not contain an edge. 

Furthermore, observe that every two non-adjacent vertices $u,v$ must have at least $b \geq \frac{k+1}{2}$ common neighbours.
Indeed, suppose $u$ has less than $b$ neighbours in $N(v)$ and consider a set $S \subseteq N(v) \cup \{u\}$ of size precisely $b$ which contains $u$ and all of its neighbours in $N(v)$, and possibly some other vertices in $N(v)$.
Now, by the assumption on the graph, there is an edge between $S$ and $N(v)\setminus S$, since the sizes of these are $b$ and $a$ respectively. However, this is a contradiction, since there are no edges between $u$ and $N(v)\setminus S$ and any edge contained in $N(v)$ creates a triangle.

To finish, recall that $G$ is non-bipartite and thus contains an odd cycle, which is then not a triangle. Further, it contains an induced odd cycle - indeed, the shortest odd cycle must be induced. Since this cycle is not a triangle, it must then contain three vertices $x,y,z$ such that $yz$ is an edge and $y,z$ are not adjacent to $x$. Since by the previous paragraph we have that both $y,z$ have at least $\frac{k+1}{2} > |N(x)|/2$ neighbours in $N(x)$, they have a common neighbour (in $N(x)$), which together with the edge $yz$ creates a triangle, a contradiction.
\end{proof} 

\noindent We will now continue with the proof \emph{assuming that} $b\geq 3$, and in the end we will deal with the few simple remaining cases when $b\leq 2$.
So assume $G$ is not isomorphic to $K_{n/2,n/2}$, so it is not bipartite.
Note that $G$ contains a $\tilde C_3$ or a $\tilde C_4$. Indeed, we get this by considering the neighbourhoods of any two vertices lying on a triangle $xyz$, whose existence is guaranteed by the previous claim; if the neighbourhoods intersect in a vertex outside of the triangle, this gives $\tilde C_3$. Otherwise, between $N(x)-\{y,z\}$ which is of size at least $k-2\geq a$, and the set $N(y)-\{x\}+\{y\}$ which is of size at least $k\geq b$, there is an edge which gives the required $\tilde C_3 $ or $\tilde C_4 $. Theorem \ref{thm:main} will now follow from the following Lemma.
\begin{lem}\label{lem:extension}
If $G$ contains a copy of $\tilde C_\ell$ for some $\ell<n-1$, then it also contains $\tilde C_{\ell+1}$ or $\tilde C_{\ell+2}$. 
\end{lem}

\noindent Indeed, to finish the proof, note that since $G$ contains $\tilde C_3$ or $\tilde C_4$, we can iteratively apply the lemma to get a family $\{\tilde C_\ell\mid \ell\in I\}$ of graphs which are all contained in $G$, such that for every pair $(i,i+1)$ of consecutive integers in $[3,n]$, one of the two is in $I$. Since each $\Tilde{C}_\ell$ contains both $C_\ell$ and $C_{\ell+1}$ we are done.

\begin{proof}[ of Lemma~\ref{lem:extension}]
Suppose for sake of contradiction that $G$ contains a $\tilde C_\ell$, but does not contain a $\tilde C_{\ell+1}$, nor a $\tilde C_{\ell+2}$. 
The central gadget we use in our proof is given by the following definition.
\begin{defn}\label{def:path with triangle} 
A $(t,s)$\emph{-switch} in $G$ is a subgraph $R$ which consists of a path $P=(1,2,\ldots,\ell+1)$ together with the vertex $x$ adjacent to vertices $t,t+1,\ldots,t+s$ with $t,s\geq 1$. 
We also write $(t,\cdot)$-switch to denote a switch for which the $s$ is not specified.

\end{defn}

\begin{figure}[ht]
    \centering
    \begin{tikzpicture}[scale=1.05,main node/.style={circle,draw,color=black,fill=black,inner sep=0pt,minimum width=3pt}]
        \tikzset{cross/.style={cross out, draw=black, fill=none, minimum size=2*(#1-\pgflinewidth), inner sep=0pt, outer sep=0pt}, cross/.default={2pt}}
	\tikzset{rectangle/.append style={draw=brown, ultra thick, fill=red!30}}

\node[main node] (a) at (0,0) [label=below:$p_1\equiv1$]{};
\node[main node] (b) at (15,0) [label=below:$p_2\equiv \ell+1$]{};
\node[main node] (x) at (4.5,1.5)[label=above:$x$]{};
\node[main node] (t1) at (4,0)[label=below:$t$]{};
\node[main node] (t2) at (5,0)[label=below:$t{+}1$]{};
\node[main node] (t3) at (6,0)[label=below:$t{+}2$]{};
\node[main node] (t4) at (7,0)[label=below:$t{+}3$]{};
     \draw[line width= 1 pt] (a) to [bend left=0](b);
     \draw[line width= 1 pt] (t1) to [bend left=0](x);
     \draw[line width= 1 pt] (t2) to [bend left=0](x);
    \draw[line width= 1 pt] (t3) to [bend left=0](x);
    \draw[line width= 1 pt] (t4) to [bend left=0](x);
	    
	    
	    

    \end{tikzpicture}
    \caption{A $(t,3)$-switch.}
    \label{fig:special set}
\end{figure}
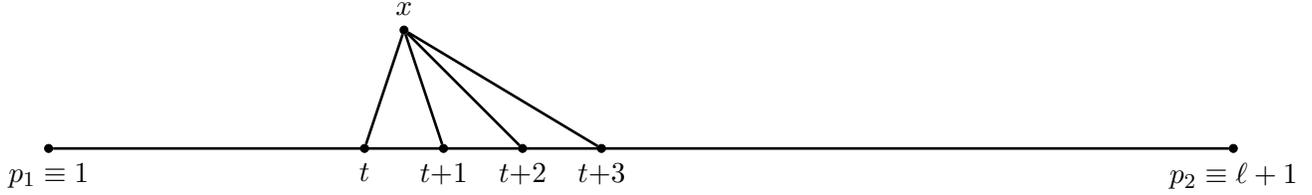
\noindent Note first that a $(t,s)$-switch exists for some $t$ and $s$. Indeed, since we have a $\tilde C_{\ell}$ and $G$ is connected, there is an edge between $\tilde C_\ell$ and a vertex $v$ outside of the $\tilde C_{\ell}$ - this evidently produces a $(t,1)$-switch whose path starts at $v$.

Let us now take a \emph{$(t,s)$-switch such that $t$ is minimized and $s$ is maximized with respect to $t$} and consider the following ordering of its vertices: $\pi=(1,2,\ldots t,x,t+1,t+2\ldots \ell+1)$, i.e. the natural order of the path $P$ with $x$ inserted between $t$ and $t+1$. Denote $p_1 :=1$ and $p_2 :=\ell+1$.
Given $v\in V(R)$ we define $v^+$ to be the vertex which comes after $v$ in the ordering $\pi$.
Given a set of vertices $T\subset V(R)$, we define $T^+$ to be the vertices obtained by shifting $T$ to the right by one, i.e., $T^+=\{v^+\mid v\in T\}$; similarly define $T^{-}$. We start with the following simple claim.

\begin{claim}\label{cl:neighbours outside}
If $t>1$, then $p_2$ has no neighbours outside of $V(R)$. If $t=1$ then $p_2$ has less than $a$ neighbours outside of $V(R)$.
\end{claim}
\begin{proof}
First, note that if $t>1$ and $p_2$ has a neighbour outside of $R$, we could add that neighbour to $R$, and remove $p_1$ from $R$, thus obtaining a $(t-1,s)$-switch, a contradiction. Now, suppose that $t=1$. Observe that $p_1$ has less than $a$ neighbours outside of $R$. Indeed, let $A = N(p_1) \setminus V(R)$ and let $T$ be the set of neighbours of $p_2-1$ in $R-\{p_2\}$, and let $T_{out}$ be the set of neighbours of $p_2-1$ outside of $R-\{p_2\}$. Then, the set $T_{out}\cup T^+$ is of size at least $\delta(G)\geq k \geq b$ and does not contain any vertices in $A$, since this creates a $\tilde C_{\ell+1}$. If $|A|\geq a$, then there is an edge $(i,j)$ between $A$ and  $T_{out}\cup T^+$, which creates either a $\tilde C_{\ell+1}$ or a $\tilde C_{\ell+2}$. 
Indeed, if $j\in T^{out}$ then obviously we get a $\tilde C_{\ell+2}$, if $j\in T^{+}\setminus\{2,x\}$ then we get a $\tilde C_{\ell+1}$ as in Fig.~\ref{fig:sub2}, and if $j=2$ then we get a $\tilde C_{\ell+1}$ whose triangle contains $1,i,2$, while if $j=x$ then we get a $\tilde C_{\ell+1}$ whose triangle contains $1,i, x$.
Hence, $|A| < a$.

To conclude, suppose that $p_2$ has at least $a$ neighbours outside $R$ and denote the set of these by $B$.
Since $p_1$ has less than $a$ neighbours outside $V(R)$ by the previous paragraph, we can take a set $T$ of at least $k-(a-1)\geq b$ neighbours of $p_1$ in $V(R)$. Hence, there is an edge $(i,j)$ between $T^-$ and $B$, which creates a $\tilde C_{\ell+2}$ (this is easy to see when $i=p_1$ or $i=x$; otherwise we get the same situation as illustrated in Fig.~\ref{fig:IIsub2}), a contradiction.
\end{proof}

\begin{claim}\label{cl:neighbours before t}
If $t>1$, then $p_2$ has no neighbours $t_0$ with $t_0<t$.
\end{claim}
\begin{proof}
Otherwise, their exists a $(t-t_0,s)$-switch, as depicted in Fig.~\ref{fig:neigh before t}, thus contradicting the optimality of $R$.
\end{proof}

\begin{figure}[ht]
    \centering
    \begin{tikzpicture}[scale=1.05,main node/.style={circle,draw,color=black,fill=black,inner sep=0pt,minimum width=3pt}]
        \tikzset{cross/.style={cross out, draw=black, fill=none, minimum size=2*(#1-\pgflinewidth), inner sep=0pt, outer sep=0pt}, cross/.default={2pt}}
	\tikzset{rectangle/.append style={draw=brown, ultra thick, fill=red!30}}

\node[main node] (a) at (0,0) [label=below:$p_1$]{};
\node[main node] (b) at (15,0) [label=below:$p_2$]{};
\node[main node] (x) at (4.5,1.5)[label=above:$x$]{};
\node[main node] (t1) at (4,0)[label=below:$t$]{};
\node[main node] (t2) at (5,0)[label=below:$t{+}1$]{};
\node[main node] (t3) at (6,0)[label=below:$t{+}2$]{};
\node[main node] (t4) at (7,0)[label=below:$t{+}3$]{};
\node[main node] (t0) at (1.5,0)[label=below:$t_0$]{};
\node[main node] (t01) at (2.5,0)[label=above:$t_0{+}1$]{};

     \draw[line width= 1 pt] (a) to [bend left=0](b);
     \draw[line width= 1 pt] (t1) to [bend left=0](x);
     \draw[line width= 1 pt] (t2) to [bend left=0](x);
    \draw[line width= 1 pt] (t3) to [bend left=0](x);
    \draw[line width= 1 pt] (t4) to [bend left=0](x);
     \draw[line width= 1 pt] (b) to [bend left=20](t0);

     \draw[red, line width= 4 pt, opacity=0.4] (t01) to [bend left=0](b);
     \draw[red, line width= 4 pt, opacity=0.4] (b) to [bend left=20](t0);
     \draw[red, line width= 4 pt, opacity=0.4] (a) to [bend left=0](t0);
    
    \end{tikzpicture}
    \caption{If $p_2$ has a neighbour before $t$ then we can use the red path to create a $(t_0,s)$-switch.}
    \label{fig:neigh before t}
\end{figure}
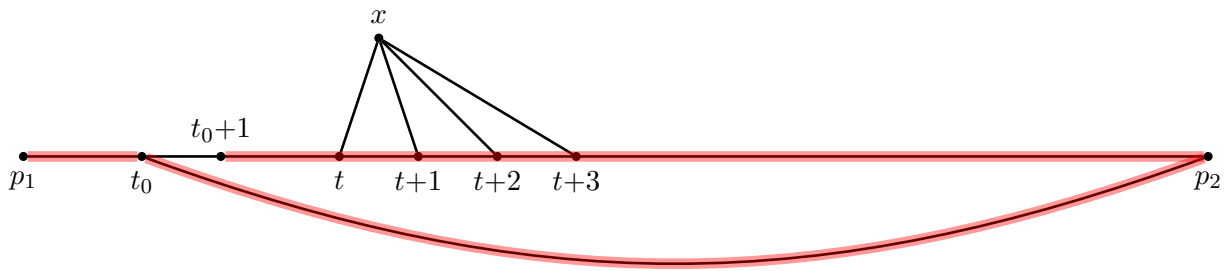

\noindent Now, define the set $S$ to consist of the last $a$ neighbours of $p_2$ in $\pi$. Observe that by Claim~\ref{cl:neighbours outside} this set exists and as usual, let $\min (S)$ denote the smallest element of $S$ in the ordering $\pi$. We then have the following.
\begin{claim}\label{cl:min S}
$\min(S)\geq t+1$.
\end{claim}
\begin{proof}
If $t=1$, note that $p_2$ is not adjacent to any of $1$ or $x$ since any such case would create a $\tilde C_{\ell+1}$, a contradiction. Therefore, $\min(S)\geq 2$. If $t>1$, then Claims~\ref{cl:neighbours outside}~and~\ref{cl:neighbours before t} imply that all of the at least $k$ neighbours of $p_2$ are in $V(R)$ and all of them are larger or equal than $t$ in $\pi$. Hence, at least $k-2 \geq a$ (recall that we are assuming that $b \geq 3$) neighbours are larger or equal than $t+1$ in $\pi$, which completes the proof.
\end{proof}
\noindent From now on, we will differentiate between two scenarios:
\begin{enumerate}[label=(\Alph*)]
    \item\label{case:p2} $p_1$ has less than $a$ neighbours in the interval $[\min(S)+1,p_2]$. 
    \item[] Then, denote by $T$ the set of neighbours of $p_1$ in $[p_1,\min(S)]$, and by $T_{out}$ the neighbours of $p_1$ outside of $V(R)$. Note that $|T| + |T_{out}|\geq k-(a-1)=b$.
    \item\label{case:p1} $p_1$ has at least $a$ neighbours in the interval $[\min(S)+1,p_2]$. 
    \item[] Then, denote this set of neighbours by $A$, denote by $T$ the set of neighbours of $p_2$ in $[p_1,\min(S)]$, by $T_{out}$ the set of neighbours of $p_2$ outside of $V(R)$. Note that by definition of $S$, $p_2$ has precisely $a-1$ neighbours in $[\min(S)+1,p_2]$ and so we have that $|T| + |T_{out}|\geq k-(a-1)=b$. Recall further that $T_{out}=\emptyset$ if $t>1$ by Claim~\ref{cl:neighbours outside}. 
\end{enumerate}

\noindent We will now consider a few cases, depending on the parameters $s$ and $t$. We will argue that besides the edges of $R$, there exist additional edges in $G[R]$ which would imply the existence of a better switch, or a copy of $\tilde C_{\ell+1}$ or $\tilde C_{\ell+2}$ in $G$, thus giving a contradiction. For example, note that $p_1$ and $p_2$ are not adjacent, since this would create a copy of $\tilde C_{\ell+1}$ in $G$. In the figures below, we give some more complex examples of edges which we may find in $G$. In the following subsections, we will consider each one of these situations and we recommend the reader to focus on the figures below only when they are referred to in the proof. We recommend reading case \ref{case:p2} in all sections first, and subsequently case \ref{case:p1} in all sections.
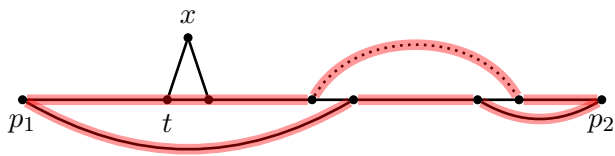
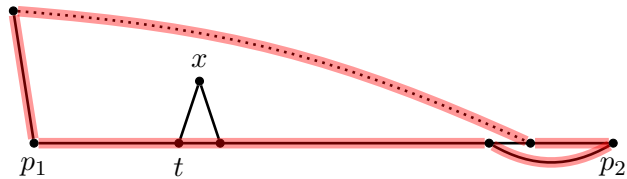
\begin{figure}[ht]
    \centering
    \begin{subfigure}{.5\textwidth}
    \begin{tikzpicture}[scale=1.1,main node/.style={circle,draw,color=black,fill=black,inner sep=0pt,minimum width=3pt}]
        \tikzset{cross/.style={cross out, draw=black, fill=none, minimum size=2*(#1-\pgflinewidth), inner sep=0pt, outer sep=0pt}, cross/.default={2pt}}
	\tikzset{rectangle/.append style={draw=brown, ultra thick, fill=red!30}}

\node[main node] (a) at (0,0) [label=below:$p_1$]{};
\node[main node] (b) at (7,0) [label=below:$p_2$]{};
\node[main node] (x) at (2,0.75)[label=above:$x$]{};
\node[main node] (t1) at (1.75,0)[label=below:$t$]{};
\node[main node] (t2) at (2.25,0)[]{};

\node[main node] (r1) at (4,0)[]{};
\node[main node] (r2) at (5.5,0)[]{};
     \draw[line width= 1 pt] (a) to [bend left=0](b);
     \draw[line width= 1 pt] (t1) to [bend left=0](x);
     \draw[line width= 1 pt] (t2) to [bend left=0](x);

        \draw[line width= 1 pt] (r1) to [bend left=30](a);
     \draw[line width= 1 pt] (r2) to [bend right=30](b);

\node[main node] (r1b) at (3.5,0)[]{};
\node[main node] (r2f) at (6,0)[]{};
  \draw[line width= 1 pt, dotted] (r1b) to [bend left=60](r2f);

     \draw[opacity=0.4,red, line width= 4 pt] (a) to [bend left=0](r1b);
     \draw[opacity=0.4,red, line width= 4 pt] (r1b) to [bend left=60](r2f);
     \draw[opacity=0.4,red, line width= 4 pt] (r2f) to [bend left=0](b);
     \draw[opacity=0.4,red, line width= 4 pt] (r2) to [bend right=30](b);
     \draw[opacity=0.4,red, line width= 4 pt] (r1) to [bend left=0](r2);
     \draw[opacity=0.4,red, line width= 4 pt] (r1) to [bend left=30](a);

    \end{tikzpicture}
     \caption{The vertex left to the neighbour of $p_1$ is adjacent \\ to the vertex to the right of the neighbour of $p_2$.}\label{fig:sub1}
   \end{subfigure}%
    \begin{subfigure}{.5\textwidth}
  \centering
\begin{tikzpicture}[scale=1.1,main node/.style={circle,draw,color=black,fill=black,inner sep=0pt,minimum width=3pt}]
        \tikzset{cross/.style={cross out, draw=black, fill=none, minimum size=2*(#1-\pgflinewidth), inner sep=0pt, outer sep=0pt}, cross/.default={2pt}}
	\tikzset{rectangle/.append style={draw=brown, ultra thick, fill=red!30}}

\node[main node] (a) at (0,0) [label=below:$p_1$]{};
\node[main node] (b) at (7,0) [label=below:$p_2$]{};
\node[main node] (x) at (2,0.75)[label=above:$x$]{};
\node[main node] (t1) at (1.75,0)[label=below:$t$]{};
\node[main node] (t2) at (2.25,0)[]{};

\node[main node] (r1) at (-0.25,1.6)[]{};
\node[main node] (r2) at (5.5,0)[]{};
     \draw[line width= 1 pt] (a) to [bend left=0](b);
     \draw[line width= 1 pt] (t1) to [bend left=0](x);
     \draw[line width= 1 pt] (t2) to [bend left=0](x);

     \draw[line width= 1 pt] (r1) to [bend left=0](a);
     \draw[line width= 1 pt] (r2) to [bend right=30](b);
   
     \draw[line width= 1 pt, dotted] (r1) to [bend left=10](r2f);

\node[main node] (r2f) at (6,0)[]{};

     \draw[opacity=0.4,red, line width= 4 pt] (a) to [bend left=0](r2);
     \draw[opacity=0.4,red, line width= 4 pt] (r1) to [bend left=10](r2f);
     \draw[opacity=0.4,red, line width= 4 pt] (r2f) to [bend left=0](b);
     \draw[opacity=0.4,red, line width= 4 pt] (r2) to [bend right=30](b);
     \draw[opacity=0.4,red, line width= 4 pt] (r1) to [bend left=0](a);

    \end{tikzpicture}
  \caption{A neighbour of $p_1$ outside $R$ is adjacent to the vertex to the right of the neighbour of $p_2$.}
  \label{fig:sub2}
\end{subfigure}
    \caption{In the first case we get a copy of $\tilde C_{\ell+1}$, and in the second a copy of $\tilde C_{\ell+2}$, whose respective cycles $C_{\ell+1}$ and $C_{\ell+2}$ are depicted in red.}
    \label{fig:first rotation}

\end{figure}

\begin{figure}[ht]
    \centering
    \begin{subfigure}{.5\textwidth}
    \begin{tikzpicture}[scale=1.1,main node/.style={circle,draw,color=black,fill=black,inner sep=0pt,minimum width=3pt}]
        \tikzset{cross/.style={cross out, draw=black, fill=none, minimum size=2*(#1-\pgflinewidth), inner sep=0pt, outer sep=0pt}, cross/.default={2pt}}
	\tikzset{rectangle/.append style={draw=brown, ultra thick, fill=red!30}}

\node[main node] (a) at (0,0) [label=below:$p_1$]{};
\node[main node] (b) at (7,0) [label=below:$p_2$]{};
\node[main node] (x) at (2,0.75)[label=above:$x$]{};
\node[main node] (t1) at (1.75,0)[label=below:$t$]{};
\node[main node] (t2) at (2.25,0)[]{};

\node[main node] (r1) at (4,0)[]{};
\node[main node] (r2) at (5.5,0)[]{};
     \draw[line width= 1 pt] (a) to [bend left=0](b);
     \draw[line width= 1 pt] (t1) to [bend left=0](x);
     \draw[line width= 1 pt] (t2) to [bend left=0](x);

        \draw[line width= 1 pt] (r2f) to [bend left=30](a);
     \draw[line width= 1 pt] (r1) to [bend left=30](b);

\node[main node] (r1b) at (3.5,0)[]{};
\node[main node] (r2f) at (6,0)[]{};
  \draw[line width= 1 pt, dotted] (r1b) to [bend right=40](r2);

     \draw[opacity=0.4,red, line width= 4 pt] (r2) to [bend left=0](r1);
     \draw[opacity=0.4,red, line width= 4 pt] (r1b) to [bend right=40](r2);
     \draw[opacity=0.4,red, line width= 4 pt] (r2f) to [bend left=0](b);
     \draw[opacity=0.4,red, line width= 4 pt] (r1) to [bend left=30](b);
     \draw[opacity=0.4,red, line width= 4 pt] (r1b) to [bend left=0](a);
     \draw[opacity=0.4,red, line width= 4 pt] (r2f) to [bend left=30](a);

    \end{tikzpicture}
     \caption{The vertex left to the neighbour of $p_1$ is adjacent \\ to the vertex to the left of the neighbour of $p_2$.} \label{fig:IIsub1}
   \end{subfigure}%
    \begin{subfigure}{.5\textwidth}
  \centering
\begin{tikzpicture}[scale=1.1,main node/.style={circle,draw,color=black,fill=black,inner sep=0pt,minimum width=3pt}]
        \tikzset{cross/.style={cross out, draw=black, fill=none, minimum size=2*(#1-\pgflinewidth), inner sep=0pt, outer sep=0pt}, cross/.default={2pt}}
	\tikzset{rectangle/.append style={draw=brown, ultra thick, fill=red!30}}

\node[main node] (a) at (0,0) [label=below:$p_1$]{};
\node[main node] (b) at (7,0) [label=below:$p_2$]{};
\node[main node] (x) at (2,0.75)[label=above:$x$]{};
\node[main node] (t1) at (1.75,0)[label=below:$t$]{};
\node[main node] (t2) at (2.25,0)[]{};

\node[main node] (r1) at (7.25,1.6)[]{};
\node[main node] (r2) at (5.5,0)[]{};
     \draw[line width= 1 pt] (a) to [bend left=0](b);
     \draw[line width= 1 pt] (t1) to [bend left=0](x);
     \draw[line width= 1 pt] (t2) to [bend left=0](x);

        \draw[line width= 1 pt] (r2f) to [bend left=30](a);
     \draw[line width= 1 pt] (r1) to [bend left=0](b);

\node[main node] (r2f) at (6,0)[]{};
  \draw[line width= 1 pt, dotted] (r1) to [bend right=10](r2);

     
     \draw[opacity=0.4,red, line width= 4 pt] (r1) to [bend right=10](r2);
     \draw[opacity=0.4,red, line width= 4 pt] (r2f) to [bend left=0](b);
     \draw[opacity=0.4,red, line width= 4 pt] (r1) to [bend left=0](b);
     \draw[opacity=0.4,red, line width= 4 pt] (r2) to [bend left=0](a);
     \draw[opacity=0.4,red, line width= 4 pt] (r2f) to [bend left=30](a);

    \end{tikzpicture}
  \caption{A neighbour of $p_2$ outside $R$ is adjacent to the vertex to the left of the neighbour of $p_1$.}
  \label{fig:IIsub2}
\end{subfigure}
    \caption{In the first case we get a copy of $\tilde C_{\ell+1}$, and in the second a copy of $\tilde C_{\ell+2}$, whose respective cycles $C_{\ell+1}$ and $C_{\ell+2}$ are depicted in red.}
    \label{fig:second rotation} 
\end{figure}

\begin{figure}[ht]
    \centering
    \begin{subfigure}{.5\textwidth}
    \begin{tikzpicture}[scale=1.1,main node/.style={circle,draw,color=black,fill=black,inner sep=0pt,minimum width=3pt}]
        \tikzset{cross/.style={cross out, draw=black, fill=none, minimum size=2*(#1-\pgflinewidth), inner sep=0pt, outer sep=0pt}, cross/.default={2pt}}
	\tikzset{rectangle/.append style={draw=brown, ultra thick, fill=red!30}}

\node[main node] (a) at (0,0) [label=below:$p_1$]{};
\node[main node] (b) at (7,0) [label=below:$p_2$]{};
\node[main node] (x) at (2,0.75)[label=above:$x$]{};
\node[main node] (t1) at (1.75,0)[label=below:$t$]{};
\node[main node] (t2) at (2.25,0)[]{};

\node[main node] (r1) at (4,0)[]{};
\node[main node] (r2) at (5.5,0)[]{};
\node[main node] (r1b) at (3.5,0)[]{};
    \node[main node] (r2f) at (6,0)[]{};
    \draw[line width= 1 pt] (a) to [bend right=0](b);
     \draw[line width= 1 pt] (t1) to [bend left=0](x);
     \draw[line width= 1 pt] (t2) to [bend left=0](x);
     \draw[line width= 1 pt] (a) to (b);
     \draw[line width= 1 pt] (a) to [bend left=52](r2);
     \draw[line width= 1 pt] (b) to [bend left=40](r1b);
    \draw[line width= 1 pt, dotted] (r1) to [bend right=40](r2f);

     
     \draw[opacity=0.4,red,line width= 4 pt] (a) to [bend left=52](r2);
     \draw[opacity=0.4,red, line width= 4 pt] (b) to [bend left=40](r1b);
      \draw[opacity=0.4,red, line width= 4 pt] (a) to [bend left=0](r1b);
       \draw[opacity=0.4,red, line width= 4 pt] (r1) to [bend left=0](r2);
        \draw[opacity=0.4,red, line width= 4 pt] (b) to [bend left=0](r2f);
     \draw[opacity=0.4,red, line width= 4 pt] (r1) to [bend right=40](r2f);

    \end{tikzpicture}
     \caption{The vertex right to the neighbour of $p_1$ is adjacent \\ to the vertex to the right of the neighbour of $p_2$.} \label{fig:IIIsub1}
   \end{subfigure}%
    \begin{subfigure}{.5\textwidth}
  \centering
\begin{tikzpicture}[scale=1.1,main node/.style={circle,draw,color=black,fill=black,inner sep=0pt,minimum width=3pt}]
        \tikzset{cross/.style={cross out, draw=black, fill=none, minimum size=2*(#1-\pgflinewidth), inner sep=0pt, outer sep=0pt}, cross/.default={2pt}}
	\tikzset{rectangle/.append style={draw=brown, ultra thick, fill=red!30}}

\node[main node] (a) at (0,0) [label=below:$p_1$]{};
\node[main node] (b) at (7,0) [label=below:$p_2$]{};
\node[main node] (x) at (2,0.75)[label=above:$x$]{};
\node[main node] (t1) at (1.75,0)[label=below:$t$]{};
\node[main node] (t2) at (2.25,0)[]{};
\node[main node] (s) at (0.5,0)[]{};

\node[main node] (r2) at (5.5,0)[]{};
\node[main node] (r2f) at (6,0)[]{};
     \draw[line width= 1 pt] (a) to [bend left=0](b);
     \draw[line width= 1 pt] (t1) to [bend left=0](x);
     \draw[line width= 1 pt] (t2) to [bend left=0](x);

        \draw[line width= 1 pt] (r2) to [bend left=30](a);
        \draw[line width= 1 pt] (r2f) to [bend right=52](a);


     
     \draw[opacity=0.4,red, line width= 4 pt] (r2f) to [bend right=52](a);
     \draw[opacity=0.4,red, line width= 4 pt] (r2f) to [bend left=0](b);
     \draw[opacity=0.4,red, line width= 4 pt] (s) to [bend left=0](r2);
     
     \draw[opacity=0.4,red, line width= 4 pt] (r2) to [bend left=30](a);

    \end{tikzpicture}
  \caption{$p_1$ is adjacent to two consecutive vertices on $P$ which are larger than $t+1$.}
  \label{fig:IIIsub2}
\end{subfigure}
    \caption{In the first case we get a copy of $\tilde C_{\ell+1}$, and in the second \\ a $(t-1,1)$-switch. The cycle $C_{\ell+1}$ and the switch are depicted in red.}
    \label{fig:third rotation} 
\end{figure}

\subsection{Triangle at the start: $t=1$}
\textbf{\ref{case:p2} holds:} 

Since $|T^-| + |T_{out}| = |T| + |T_{out}| \geq b$ and $|S^{+}| = |S| = a$, there is an edge between $T^-\cup T_{out}$ and $S^+$. This creates either a $\tilde C_{\ell+1}$ or a $\tilde C_{\ell+2}$ (see Fig.~\ref{fig:first rotation}), so we are done.
\\
\noindent \textbf{\ref{case:p1} holds:} 
Let $T_0=T-\{2\}$. Note that none of $p_1$ and $x$ are in $T$ since otherwise a $\Tilde C_{\ell+1}$ exists, and thus, $\min(T_0) > 2$. Now, $T_0^{-}\cup T_{out}\cup\{p_2\}$ is of size at least $b$ and so, there is an edge between this set and $A^{-}$, which always creates either a $\tilde C_{\ell+1}$ or a $\tilde C_{\ell+2}$ (see Fig.~\ref{fig:second rotation}).

\subsection{Triangle starts at the second vertex: $t=2$}

\textbf{\ref{case:p2} holds:} Note that $x\notin T$ since otherwise there would exist a $(1,\cdot)$-switch, starting with the triangle $(1,x,2)$. Consider then the set $T^-\cup T_{out}$ which is of size at least $b$. Between $T^-\cup T_{out}$ and $S^+$ (which is of size $a$) there is an edge $(i,j)$, which creates a $\tilde C_{\ell+1}$ or a $\tilde C_{\ell+2}$. Indeed, if $i\neq x$ then we can proceed as in Fig.~\ref{fig:sub1}, and if $i=x$ we proceed as below (Fig.~\ref{fig:t=2}).

\begin{figure}[ht]
    \centering
    \begin{tikzpicture}[scale=1.1,main node/.style={circle,draw,color=black,fill=black,inner sep=0pt,minimum width=3pt}]
        \tikzset{cross/.style={cross out, draw=black, fill=none, minimum size=2*(#1-\pgflinewidth), inner sep=0pt, outer sep=0pt}, cross/.default={2pt}}
	\tikzset{rectangle/.append style={draw=brown, ultra thick, fill=red!30}}

\node[main node] (a) at (-0.25,0) [label=below:$1$]{};
\node[main node] (b) at (7,0) [label=below:$p_2$]{};
\node[main node] (x) at (0.5,0.75)[label=above:$x$]{};
\node[main node] (t1) at (0.25,0)[label=above:]{};
\node[main node] (t2) at (0.75,0)[label=below:]{};

\node[main node] (r2) at (5.5,0)[]{};
     \draw[line width= 1 pt] (a) to [bend left=0](b);
     \draw[line width= 1 pt] (t1) to [bend left=0](x);
     \draw[line width= 1 pt] (t2) to [bend left=0](x);

     \draw[line width= 1 pt] (a) to [bend right=60](t2);
     \draw[line width= 1 pt] (r2) to [bend left=40](b);

\node[main node] (r2f) at (6,0)[]{};
  \draw[line width= 1 pt, dotted] (r2f) to [bend right=40](x);

     \draw[opacity=0.4,red, line width= 4 pt] (t1) to [bend left=0](r2);
     \draw[opacity=0.4,red, line width= 4 pt] (r2f) to [bend right=40](x);
     \draw[opacity=0.4,red, line width= 4 pt] (t1) to [bend left=0](x);
     \draw[opacity=0.4,red, line width= 4 pt] (r2) to [bend left=40](b);
     \draw[opacity=0.4,red, line width= 4 pt] (r2f) to [bend left=0](b);

    \end{tikzpicture}
     \caption{The case of $i = x$. We get a $\tilde C_{\ell+1}$ where the triangle consists of vertices $1,2,3$.} \label{fig:t=2}

\end{figure}
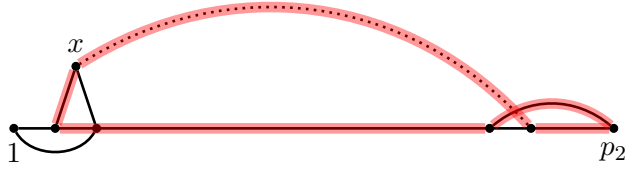

\noindent\textbf{\ref{case:p1} holds:} Recall that $|T|\geq b$ since $T_{out} = \emptyset$, and that $p_1 \notin T$. If not both edges $(x,p_2)$ and $(3,p_2)$ are present, then for each vertex (if it is in $T$) we can assign a unique vertex as follows: $2\rightarrow{}1$, $x\rightarrow{} 4$, and for all other vertices $v\rightarrow{} v+1$. If $T^*$ is the set of vertices assigned to $T$, then there is an edge $(i,j)$ between $T^*$ and $A^+$ (note that these are disjoint since $\min(A)>\max(T)=\min (S)\geq t+1=3$, implying that all elements in $A^+$ are larger than $\max({T^*})$). This edge always creates a copy of $\tilde C_{\ell+1}$ when $i\neq 1$ (as in Fig.~\ref{fig:IIIsub1}), and when $i=1$ then it creates a $(1,\cdot)$-switch (as in Fig.~\ref{fig:IIIsub2}), a contradiction.

Otherwise, if both $(x,p_2)$ and $(3,p_2)$ exist, then it must be that $s \geq 2$, since these edges can be used to form a $(t,2)$-switch (see Fig.~\ref{fig:s is bigger}).

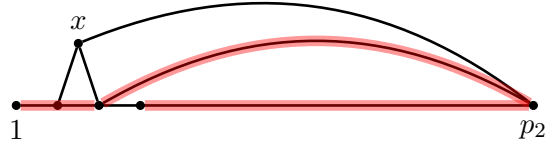
\begin{figure}[ht]
    \centering
    \begin{tikzpicture}[scale=1.1,main node/.style={circle,draw,color=black,fill=black,inner sep=0pt,minimum width=3pt}]
        \tikzset{cross/.style={cross out, draw=black, fill=none, minimum size=2*(#1-\pgflinewidth), inner sep=0pt, outer sep=0pt}, cross/.default={2pt}}
	\tikzset{rectangle/.append style={draw=brown, ultra thick, fill=red!30}}

\node[main node] (a) at (-0.25,0) [label=below:$1$]{};
\node[main node] (b) at (6,0) [label=below:$p_2$]{};
\node[main node] (x) at (0.5,0.75)[label=above:$x$]{};
\node[main node] (t1) at (0.25,0)[label=above:]{};
\node[main node] (t2) at (0.75,0)[label=below:]{};

\node[main node] (r2) at (1.25,0)[]{};
     \draw[line width= 1 pt] (a) to [bend left=0](b);
     \draw[line width= 1 pt] (t1) to [bend left=0](x);
     \draw[line width= 1 pt] (t2) to [bend left=0](x);

     \draw[line width= 1 pt] (x) to [bend right=-30](b);
     \draw[line width= 1 pt] (t2) to [bend left=30](b);

     \draw[opacity=0.4,red, line width= 4 pt] (a) to [bend left=0](t2);
     \draw[opacity=0.4,red, line width= 4 pt] (r2) to [bend left=0](b);
     \draw[opacity=0.4,red, line width= 4 pt] (t2) to [bend left= 30](b);

    \end{tikzpicture}
     \caption{Obtaining a $(t,2)$-switch, when $p_2$ is adjacent to both $x$ and $3$.} \label{fig:s is bigger}

\end{figure}

\noindent Now we define $T^*$ differently as follows: $x\rightarrow{} p_1$ and $v\rightarrow{}v+1$ for all other $v$ in $T$, and we can proceed as before, by finding an edge $(i,j)$ between $T^*$ and $A^+$. Crucially, note that if now $i=3$, after doing rotation (as in Fig.~\ref{fig:IIIsub1}), we do destroy the triangle $2,x,3$, but the triangle $3,x,4$ is preserved, so it was crucial that $s\geq 2$.

\subsection{Triangle in the middle: $t>2$} 
 
\textbf{\ref{case:p2} holds:}
First, note that $p_1$ is not adjacent to both $t$ and $x$, and $p_1$ is not adjacent to both $t$ and $t-1$, as in both cases we get a better switch (see Fig.~\ref{fig:p_1 adjacency}).

\begin{figure}[ht]
    \centering
    \begin{subfigure}{.5\textwidth}
    \begin{tikzpicture}[scale=1.1,main node/.style={circle,draw,color=black,fill=black,inner sep=0pt,minimum width=3pt}]
        \tikzset{cross/.style={cross out, draw=black, fill=none, minimum size=2*(#1-\pgflinewidth), inner sep=0pt, outer sep=0pt}, cross/.default={2pt}}
	\tikzset{rectangle/.append style={draw=brown, ultra thick, fill=red!30}}

\node[main node] (a) at (-0.25,0) [label=below:$p_1$]{};
\node[main node] (b) at (7,0) [label=below:$p_2$]{};
\node[main node] (x) at (2.75,0.75)[label=above:$x$]{};
\node[main node] (t1) at (2.5,0)[label=below:]{};
\node[main node] (t2) at (3,0)[]{};

\node[main node] (r) at (2,0)[]{};
     \draw[line width= 1 pt] (a) to [bend left=0](b);
     \draw[line width= 1 pt] (t1) to [bend left=0](x);
     \draw[line width= 1 pt] (t2) to [bend left=0](x);

     \draw[line width= 1 pt] (a) to [bend right=-10](x);
     \draw[line width= 1 pt] (a) to [bend left=20](t1);
   

      \draw[opacity=0.4,red, line width= 4 pt] (a) to [bend right=0](r);
     \draw[opacity=0.4,red, line width= 4 pt] (a) to [bend right=-10](x);
     \draw[opacity=0.4,red, line width= 4 pt] (x) to [bend left=0](t2);
     \draw[opacity=0.4,red, line width= 4 pt] (t2) to [bend left=0](b);
     \draw[opacity=0.4,blue, line width= 4 pt] (x) to [bend left=0](t1);
     \draw[opacity=0.4,blue, line width= 4 pt] (a) to [bend left=20](t1);

    \end{tikzpicture}
     \caption{The $(t-1,\cdot)$-switch where the path is red,\\and two blue edges to complete the triangle.} \label{fig:p2b}
   \end{subfigure}%
    \begin{subfigure}{.5\textwidth}
  \centering
  \begin{tikzpicture}[scale=1.1,main node/.style={circle,draw,color=black,fill=black,inner sep=0pt,minimum width=3pt}]
        \tikzset{cross/.style={cross out, draw=black, fill=none, minimum size=2*(#1-\pgflinewidth), inner sep=0pt, outer sep=0pt}, cross/.default={2pt}}
	\tikzset{rectangle/.append style={draw=brown, ultra thick, fill=red!30}}

\node[main node] (a) at (-0.25,0) [label=below:$p_1$]{};
\node[main node] (b) at (7,0) [label=below:$p_2$]{};
\node[main node] (x) at (2.75,0.75)[label=above:$x$]{};
\node[main node] (t1) at (2.5,0)[label=below:]{};
\node[main node] (t2) at (3,0)[]{};

\node[main node] (r) at (2,0)[]{};
\node[main node] (r1) at (1.5,0)[]{};
     \draw[line width= 1 pt] (a) to [bend left=0](b);
     \draw[line width= 1 pt] (t1) to [bend left=0](x);
     \draw[line width= 1 pt] (t2) to [bend left=0](x);

     \draw[line width= 1 pt] (a) to [bend left=30](r);
     \draw[line width= 1 pt] (a) to [bend left=50](t1);
   

      \draw[opacity=0.4,red, line width= 4 pt] (a) to [bend right=0](r1);
     \draw[opacity=0.4,red, line width= 4 pt] (t1) to [bend left=0](x);
     \draw[opacity=0.4,red, line width= 4 pt] (x) to [bend left=0](t2);
     \draw[opacity=0.4,red, line width= 4 pt] (t2) to [bend left=0](b);
     \draw[opacity=0.4,blue, line width= 4 pt] (r) to [bend left=0](t1);
     \draw[opacity=0.4,red, line width= 4 pt] (a) to [bend left=50](t1);
   \draw[opacity=0.4,blue, line width= 4 pt] (a) to [bend left=30](r);

    \end{tikzpicture}
  \caption{The $(t-2,\cdot)$-switch where the path is red,\\and two blue edges to complete the triangle.}
  \label{fig:p1b}
\end{subfigure}
    \caption{$p_1$ is not adjacent to both $t$ and $x$, and $p_1$ is not adjacent to both $t$ and $t-1$ as in both cases we create a better switch.}
    \label{fig:p_1 adjacency} 
\end{figure}
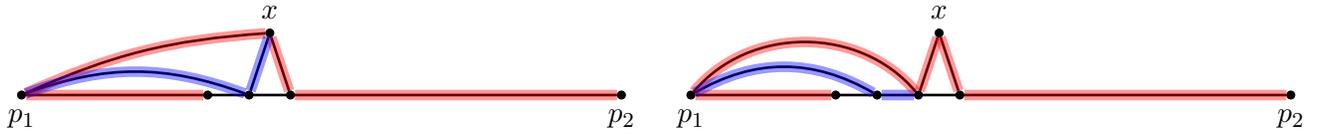

\noindent Now, to each vertex in $T$ we assign a unique vertex as follows, depending on the adjacencies between $p_1$ and the set $Q=\{t,t+1,x\}$:
\begin{enumerate}[label=(\roman*)]
\itemsep0em 
    \item If $p_1$ is adjacent to at most one vertex in $Q$, then assign: $x\rightarrow{}t-1$, $t+1\rightarrow{}t-1$ and $v\rightarrow{} v^-$ for all other vertices in $T$.
    \item If $p_1$ is adjacent to only $x,t+1$ in $Q$, then: $x\rightarrow{}t$, $t+1\rightarrow{}t-1$ and $v\rightarrow{} v^-$ for all other vertices in $T$.
     \item If $p_1$ is adjacent only to $t,t+1$ in $Q$, then by the observation above it is not adjacent to $t-1$. We then take: $t\rightarrow{}t-1$, $t+1\rightarrow{}t-2$ and $v\rightarrow{} v^-$ for all other vertices in $T$.
\end{enumerate}
As shown before, $p_1$ cannot be adjacent to both $x,t$ and so, one of the options above must hold. Let $T^*$ be the set of assigned vertices, and note that $|T^*| = |T|$ and thus $|T^* \cup T_{out} | \geq b$. Hence we have an edge $(i,j)$ between $T^*$ and $S^+$, and we can check that in each case we either get a $(t',s')$-switch with some $t' < t$ or a $\tilde C_{\ell+1}$. 

Indeed, for (i) we have a situation as depicted in Fig.~\ref{fig:B} if $i=t-1$, and Fig.~\ref{fig:first rotation} otherwise. For (ii) we have a situation as in Fig.~\ref{fig:C} if $i=t$, as in Fig.~\ref{fig:B} if $i=t-1$ and otherwise we have again the situation in Fig.~\ref{fig:first rotation}. For (iii) we have the situation of Fig.~\ref{fig:B} if $i \in \{t-1,t-2\}$ and the situation of Fig.~\ref{fig:first rotation} otherwise.

\begin{figure}[ht]
    \centering
    \begin{tikzpicture}[scale=1.1,main node/.style={circle,draw,color=black,fill=black,inner sep=0pt,minimum width=3pt}]
        \tikzset{cross/.style={cross out, draw=black, fill=none, minimum size=2*(#1-\pgflinewidth), inner sep=0pt, outer sep=0pt}, cross/.default={2pt}}
	\tikzset{rectangle/.append style={draw=brown, ultra thick, fill=red!30}}

\node[main node] (a) at (-0.25,0) [label=below:$p_1$]{};
\node[main node] (b) at (7,0) [label=below:$p_2$]{};
\node[main node] (x) at (2.75,0.75)[label=above:$x$]{};
\node[main node] (t1) at (2.5,0)[label=below:]{};
\node[main node] (t2) at (3,0)[]{};

\node[main node] (r2) at (5.5,0)[]{};
\node[main node] (r2f) at (6,0)[]{};
\node[main node] (r1) at (1,0)[]{};
\node[main node] (r1f) at (1.5,0)[]{};
     \draw[line width= 1 pt] (a) to [bend left=0](b);
     \draw[line width= 1 pt] (t1) to [bend left=0](x);
     \draw[line width= 1 pt] (t2) to [bend left=0](x);

     \draw[line width= 1 pt] (b) to [bend right=60](r2);
     \draw[line width= 1 pt, dotted] (r2f) to [bend left=40](r1);
   

      \draw[opacity=0.4,red, line width= 4 pt] (b) to [bend right=60](r2);
     \draw[opacity=0.4,red, line width= 4 pt] (r2f) to [bend left=40](r1);
     \draw[opacity=0.4,red, line width= 4 pt] (a) to [bend left=0](r1);
     \draw[opacity=0.4,red, line width= 4 pt] (r2f) to [bend left=0](b);
     \draw[opacity=0.4,red, line width= 4 pt] (r1f) to [bend left=0](r2);

    \end{tikzpicture}
     \caption{The red line represents the path of a switch with the triangle closer to $p_1$} \label{fig:B}
\end{figure}

\begin{figure}[h]
    \centering
    \begin{tikzpicture}[scale=1.1,main node/.style={circle,draw,color=black,fill=black,inner sep=0pt,minimum width=3pt}]
        \tikzset{cross/.style={cross out, draw=black, fill=none, minimum size=2*(#1-\pgflinewidth), inner sep=0pt, outer sep=0pt}, cross/.default={2pt}}
	\tikzset{rectangle/.append style={draw=brown, ultra thick, fill=red!30}}

\node[main node] (a) at (-0.25,0) [label=below:$p_1$]{};
\node[main node] (b) at (7,0) [label=below:$p_2$]{};
\node[main node] (x) at (2.75,0.75)[label=above:$x$]{};
\node[main node] (t1) at (2.5,0)[label=below:]{};
\node[main node] (t2) at (3,0)[]{};

\node[main node] (r2) at (5.5,0)[]{};
\node[main node] (r2f) at (6,0)[]{};
\node[main node] (r1) at (2,0)[]{};
\node[main node] (r) at (3.5,0)[]{};

     \draw[line width= 1 pt] (a) to [bend left=0](b);
     \draw[line width= 1 pt] (t1) to [bend left=0](x);
     \draw[line width= 1 pt] (t2) to [bend left=0](x);

     \draw[line width= 1 pt] (b) to [bend right=60](r2);
     \draw[line width= 1 pt] (a) to [bend right=60](t2);
     \draw[line width= 1 pt] (a) to [bend right=0](x);
     \draw[line width= 1 pt, dotted] (r2f) to [bend left=40](t1);
   

      \draw[opacity=0.4,red, line width= 4 pt] (b) to [bend right=60](r2);
     \draw[opacity=0.4,red, line width= 4 pt] (r2f) to [bend left=40](t1);
     \draw[opacity=0.4,blue, line width= 4 pt] (a) to [bend right=60](t2);
     \draw[opacity=0.4,red, line width= 4 pt] (a) to [bend right=0](x);
     \draw[opacity=0.4,red, line width= 4 pt] (a) to [bend left=0](r1);
      \draw[opacity=0.4,red, line width= 4 pt] (x) to [bend left=0](t1);
     \draw[opacity=0.4,red, line width= 4 pt] (r2f) to [bend left=0](b);
      \draw[opacity=0.4,red, line width= 4 pt] (r2) to [bend left=0](r);
       \draw[opacity=0.4,blue, line width= 4 pt] (x) to [bend right=0](t2);

    \end{tikzpicture}
     \caption{The red line represents the path of a switch with the (blue) triangle closer to $p_1$} \label{fig:C}
\end{figure}
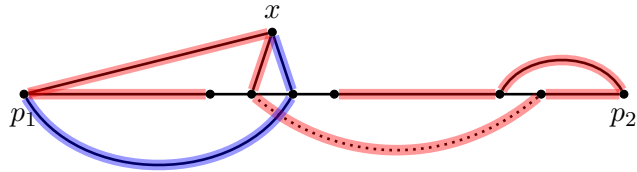

\noindent\textbf{\ref{case:p1} holds:}
Recall that by Claim~\ref{cl:neighbours before t}, the vertex $p_2$ has no neighbours before $t$ in the ordering $\pi$ and that $T_{out} = \emptyset$. To each vertex in $T$ we can then assign a unique vertex as follows: $x\rightarrow t-2$, $t\rightarrow t-1$ and $v\rightarrow v+1$ for all other $v\in T$ (which must have $v\geq t+1$). Let $T^*$ be the set of assigned vertices, and note then that $|T^*| = |T|\geq b$. Hence we have an edge $(i,j)$ between $T^*$ and $A^+$, which are disjoint, as we already explained in Section 2.2, Part (B). In each case we either get a $(t',\cdot)$-switch with $t' < t$ or we create a $\tilde C_{\ell+1}$. 
Indeed, if $i=t-1$ or $i=t-2$ then we are done by Fig~\ref{fig:G}, otherwise we are done by Fig.~\ref{fig:IIIsub1}.

\begin{figure}[h]
    \centering
    \begin{tikzpicture}[scale=1.1,main node/.style={circle,draw,color=black,fill=black,inner sep=0pt,minimum width=3pt}]
        \tikzset{cross/.style={cross out, draw=black, fill=none, minimum size=2*(#1-\pgflinewidth), inner sep=0pt, outer sep=0pt}, cross/.default={2pt}}
	\tikzset{rectangle/.append style={draw=brown, ultra thick, fill=red!30}}

\node[main node] (a) at (-0.25,0) [label=below:$p_1$]{};
\node[main node] (b) at (7,0) [label=below:$p_2$]{};
\node[main node] (x) at (2.75,0.75)[label=above:$x$]{};
\node[main node] (t1) at (2.5,0)[label=below:]{};
\node[main node] (t2) at (3,0)[]{};

\node[main node] (r2) at (5.5,0)[]{};
\node[main node] (r2f) at (6,0)[]{};
\node[main node] (r1) at (1,0)[]{};
\node[main node] (r1f) at (1.5,0)[]{};
     \draw[line width= 1 pt] (a) to [bend left=0](b);
     \draw[line width= 1 pt] (t1) to [bend left=0](x);
     \draw[line width= 1 pt] (t2) to [bend left=0](x);

     \draw[line width= 1 pt] (a) to [bend right=50](r2);
     \draw[line width= 1 pt, dotted] (r2f) to [bend right=60](r1);
   

      \draw[opacity=0.4,red, line width= 4 pt] (a) to [bend right=50](r2);
     \draw[opacity=0.4,red, line width= 4 pt] (r2f) to [bend right=60](r1);
     \draw[opacity=0.4,red, line width= 4 pt] (a) to [bend left=0](r1);
     \draw[opacity=0.4,red, line width= 4 pt] (r2f) to [bend left=0](b);
     \draw[opacity=0.4,red, line width= 4 pt] (r1f) to [bend left=0](r2);

    \end{tikzpicture}
     \caption{The red line represents the path of a switch with the triangle closer to $p_1$} \label{fig:G}
\end{figure}

\noindent This completes the proof of Lemma~\ref{lem:extension}.
\end{proof}
\subsection{Completing the proof: $b\leq 2$}
When $a\leq b\leq 2$, the proof is significantly shorter.
Indeed, this implies that there is an edge between any two disjoint sets of size at least $2$ in $G$. If $|G|\leq 7$ one can check by hand that the statement holds and we leave this as an exercise to the reader (one can already assume that $G$ is Hamiltonian, as guaranteed by Theorem~\ref{thm:mcdiarmid}, and then analyse what cycles are created by adding edges between pairs of vertices in the Hamilton cycle).

Otherwise, first note that by Theorem \ref{thm:mcdiarmid}, $G$ contains a Hamilton cycle. We will now show that if $G$ contains a cycle of length $\ell$ with $6\leq \ell\leq n-2$, then it contains a cycle of length $\ell+1$. This reduces pancyclicity to only finding cycles of length $3,4,5$ and $6$ in $G$. Consider then a cycle $C_\ell$ in $G$ and suppose for sake of contradiction that there is no $C_{\ell + 1}$. Let $x,y$ be two vertices outside of the cycle $C_\ell$. Trivially, since $l \geq 6$, it must be that at least one of $x,y$ has at least $3$ neighbours in $C_\ell$ (otherwise there would be two vertices in $C_\ell$ not adjacent to $\{x,y\}$, contradicting the assumption on $G$). Without loss of generality, assume that $x$ is adjacent to $z_1,z_2,z_3 \in C_\ell$. If any pair of vertices $z_i,z_j$ is consecutive in the cycle $C_\ell$, then we can extend this cycle using $x$ to create a $C_{\ell + 1}$. Otherwise, fix some orientation of $C_{\ell}$ and for each $v \in C_{\ell}$, denote by $v^-$ the vertex before $v\in C_\ell$ in this orientation. By assumption, there is then an edge between the sets $\{z_1^-,x\}$ and $\{z_2^-,z_3^-\}$. In turn, it is easy to check that any such edge creates a cycle $C_{\ell+1}$ on the vertex set of $C_\ell + x$.

We now prove the existence of cycles of lengths from $3,4,5,6$. Note that we have already shown that $G$ contains a triangle $xyz$ in Claim \ref{claim:triangle}. Now, if two vertices on the triangle have a common neighbour outside then we also have a $C_4$. Otherwise, note that since by assumption there is an edge between every two disjoint sets of size two, it must be that every pair of vertices has a vertex with degree at least $\frac{n-3}{2} > 2$ - therefore, two vertices in the triangle have a disjoint neighbourhood of size at least $2$ outside the triangle. Between those two neighbourhoods there is an edge, which again gives a $C_4$ (and a $C_5$). If $G$ does not have a $C_5$, then we are in the former case with two triangles sharing an edge ($C_4$ with a diagonal). Now, all except one vertex outside of these four vertices have at least $2$ neighbours inside of it. The only way not to create a $C_5$ is to have all of these (at least 2) vertices be adjacent only to the two vertices of degree $3$. But then there is no edge
between those vertices outside, and the remaining vertices of our $C_4.$

Finally, if we have a $C_5$ then again all but at most one vertex outside of it, are adjacent to at least two vertices in $C_5$. One can check that since we have at least $2$ of them, this always gives a $C_6$ as well. This completes our proof.

\end{proof}
\section{Concluding remarks}\label{sec:concluding remarks}
Bondy's meta-conjecture states that every non-trivial condition which implies Hamiltonicity, also implies pancyclicity, up to a certain small collection of exceptional graphs. Clearly, there are some cases of natural Hamiltonicity conditions for which this statement fails. For example, it is well known (see \cite{KS-03}) that pseudorandom graphs are Hamiltonian, but even rather dense pseudo-random graphs might have no short cycles to be pancyclic. On the other hand, in addition to the results presented in this paper, we know by now that several well-known Hamiltonicity theorems can be extended to give pancyclicity, for example see \cite{bondy1971pancyclic, bauer1990hamiltonian, ourJOpaper}. Hence, it would be interesting to explore other interesting Hamiltonicity conditions and understand whether they indeed imply pancyclicity.

\end{document}